\documentclass[11pt,amsfonts]{article}
\usepackage{amsthm,amsmath,amssymb,amsfonts,mathrsfs,microtype}
\usepackage[usenames,dvipsnames]{color}
\usepackage{graphicx}
\usepackage{calrsfs}
\usepackage{hyperref}

\numberwithin{equation}{section}
\newtheorem{thm}{Theorem}[section]
\newtheorem{pro}{Proposition}[section]
\newtheorem{cor}{Corollary}[section]
\newtheorem{lmm}{Lemma}[section]

\theoremstyle{definition}

\newtheorem{dfn}{Definition}[section]
\newtheorem{exa}{Example}[section]
\newtheorem{rem}{Remark}[section]

\newcommand{\B}{\mathcal{B}}
\renewcommand{\H}{\mathcal{H}}

\renewcommand{\L}{\mathcal{L}}

\newcommand{\CM}{\mathcal{R}}

\newcommand{\HH}{\mathcal{H}}

\newcommand{\cS}{\mathcal{S}}

\newcommand{\E}{\mathbb{E}}

\renewcommand{\P}{\mathbb{P}}

\newcommand{\D}{\mathrm{D}}
\newcommand{\bD}{\mathbf{D}}
\newcommand{\I}{\mathrm{I}}
\newcommand{\K}{\mathrm{K}}

\newcommand{\R}{\mathrm{R}}

\newcommand{\Tr}{\mathrm{Tr}}

\newcommand{\T}{\mathrm{T}}

\renewcommand{\d}{\mathrm{d}}
\newcommand{\e}{\mathrm{e}}

\newcommand{\la}{\langle}
\newcommand{\ra}{\rangle}
\newcommand{\F}{\mathcal{F}}
\newcommand{\1}{\mathbf{1}}
\renewcommand{\i}{\mathrm{i}}
\newcommand{\dom}{\mathrm{dom}}

\begin{document}


\title{{\bf Integration-by-Parts Characterizations of Gaussian Processes} }

\author{{\sc Ehsan Azmoodeh}\\
Ruhr-Universit\"at Bochum \\
IB 2/101 \\
GER-44780 Bochum  \\
GERMANY
%
\and 
{\sc Tommi Sottinen}\\
University of Vaasa \\
P.O. Box 700 \\
FIN-65101 Vaasa \\
FINLAND 
%
\and
{\sc Ciprian A. Tudor} \\
University of Lille 1 \\
CNRS, UMR 8524\\ Laboratoire Paul Painlev\'e\\
Cit\'e Scientifique,  B\^at. M3\\
F-59655 Villeneuve d'Ascq \\
FRANCE
\and 
{\sc Lauri Viitasaari}\\
University of Helsinki\\
P.O. Box 68 \\
FIN-00014 Helsingin yliopisto \\
FINLAND
}

\date{\today}

\maketitle

\begin{abstract}
The Malliavin integration-by-parts formula is a key ingredient to develop stochastic analysis on the Wiener space. In this article we show that a suitable integration-by-parts formula also characterizes a wide class of Gaussian processes, the so-called Gaussian Fredholm processes. 
\end{abstract}



\section{Introduction}

It is well-known that the law of a standard normal random variable $X$ is fully characterized by the Stein's equation (also known as the integration-by-parts formula)
\begin{equation}
\label{1}
\E \left[f'(X)\right] =\E\left[ X f(X)\right]. 
\end{equation}
More exactly, $X$ follows the standard normal distribution if and only if for any function $f\colon\mathbb{R} \to\mathbb{R}$ that is integrable with respect to the standard Gaussian measure on $\mathbb{R}$, the relation (\ref{1})
holds true. 
The formula (\ref{1}) can be extended to finite-dimensional Gaussian vectors and it can be also expressed in terms of the Malliavin calculus in various ways (see e.g. Hsu \cite{Hsu} or Nourdin and Peccati \cite{Nourdin-Peccati-2012}).

Our purpose is to  prove an integration-by-parts formula that characterizes (centered) Gaussian stochastic processes.  The framework is to view the stochastic processes as random paths on $\L^2=\L^2([0,1])$, and to show that the law $\P=\P^X$ of the co-ordinate process $X$ satisfies a certain integration-by-parts formula depending on a covariance function $R$ if and only if under $\P$ it is a centered Gaussian process with the covariance function $R$. 
The line of attack is to use the Fredholm representation of $\L^2$-valued Gaussian processes provided in \cite{Sottinen-Viitasaari-2015} and \cite{Sottinen-Viitasaari-2014-preprint}.  

On related research we mention Barbour \cite{Barbour-1990}, Coutin and Decreusefond \cite{Cou-Dec}, Kuo and Lee \cite{Kuo-Lee-2011}, Shih \cite{Shih-2011}, and Sun and Guo \cite{Sun-Guo-2015}. In particular, we note that Theorem 3.1 of Shih \cite{Shih-2011} characterizes Gaussian measures on Banach spaces via the following integration-by-parts formula (the formulation uses the machinery of abstract Wiener spaces $(i,\CM,\B)$, where $i\colon\CM\to\B$ is the canonical embedding and $\CM$ is the Cameron--Martin space of an $\B$-valued Gaussian random variable): 
Let $X$ be a $\B$-valued random variable. Then $\P$ is Gaussian if and only if
$$
\E\left[{\la X, \D f(X)\ra}_{\B,\B^*}\right] = \E\left[{\Tr}_\CM \D^2 f(X)\right]
$$ 
for all $f:\B\to\mathbb{R}$ such that $\D^2f(X)$ is trace-class on $\CM$, where $\D$ denotes the Gross derivative. In Section \ref{sect:discussion}, we will discuss the connection between our integration-by-parts formula  \eqref{eq:ibp-gfp} and results in Shih \cite{Shih-2011}.  In particular, we will show than if one works on the smaller space of continuous functions on $[0,1]$ vanishing at the origin, then  our formula and  Shih's results are different. On the other hand,  our approach  does not use the Cameron--Martin space  and the abstract Wiener space structure. Instead, we work solely on the fixed path-space $\L^2=\L^2([0,1])$  and this makes somehow our characterization of Gaussian measures simpler.

The rest of the paper is organized as follows. In Section \ref{sec:preliminaries} we recall some preliminaries including operators on $\L^2$ spaces and associated Gaussian Fredholm processes. We also define ``pathwise'' Malliavin derivative that is crucial for our results. Indeed, as the underlying processes in Section \ref{sec:IBP} are not Gaussian \emph{a priori}, the Malliavin derivative cannot be defined in a traditional sense using Gaussian spaces. We formulate and prove our integration-by-parts characterizations in Section \ref{sec:IBP}. We also provide several examples including characterizations for Brownian motion and Brownian bridges. We end the paper with a short discussion in Section \ref{sect:discussion} on the links between our findings and the characterization of Gaussian processes in Shih \cite{Shih-2011}.
\section{Preliminaries}
\label{sec:preliminaries}
We recall some necessary preliminaries in order to prove our results. 
We begin by considering kernels and operators in $\L^2=\L^2([0,1])$ which are then connected with $\L^2$-valued centered Gaussian processes via the so-called Fredholm representation. For details of the facts and constructions related to Fredholm processes, we refer to \cite{Sottinen-Viitasaari-2015}, \cite{Sottinen-Viitasaari-2014-preprint}, and \cite{Sottinen-Yazigi-2014}.

\subsection{Some Operators and Kernels}

Let $\1_A$ denote the indicator of a set $A$.  We use the short-hand $\1_t = \1_{[0,t)}$, which is related with the interpretation $\int_a^b = \int_{[a,b)}$. 
Recall that $\L^2=\L^2([0,1])$.  We use the identification $\L^2\times\L^2 = \L^2([0,1]^2)$. We denote by $\mathcal{E}\subset\L^2$ the set of right-continuous step-functions. 

\begin{dfn}[Associated Operator]\label{dfn:T}
For a kernel $T\in\L^2\times\L^2$ its \emph{associated operator} $\T\colon\L^2\to\L^2$ is defined as
$$
\T f(t) = \int_0^1 f(s)T(t,s)\, \d s.
$$
\end{dfn}

We note that the associated operator $\T\colon\L^2\to\L^2$ is bounded.  Indeed, by the Cauchy--Schwarz inequality
\begin{eqnarray*}
{\|\T f\|}_{\L^2}^2 &=&
\int_0^1\left(\int_0^1 f(s) T(t,s)\, \d s\right)^2\, \d t \\
&\le& 
\int_0^1 \left(\int_0^1 f(s)^2\, \d s\right)\left(\int_0^1 T(t,s)^2\, \d s\right)\,\d t \\
&=& \int_0^1 f(s)^2\, \d s\, \int_0^1\!\!\!\int_0^1 T(t,s)^2 \, \d s \d t\\
&=& {\|f\|}_{\L^2}^2 {\| T \|}_{\L^2\times\L^2}^2.
\end{eqnarray*}

\begin{exa}\label{exa:indicator_kernel}
For the indicator kernel
$
I(t,s) = \1_t(s)
$
the associated operator is just the definite integral
$$
\I f(t) = \int_0^t f(s)\, \d s.
$$
\end{exa}

\begin{dfn}[Associated Adjoint Operator]
For a kernel $T\in\L^2\times\L^2$ its \emph{associated adjoint operator} $\T^*$ is defined by extending linearly the relations
$$
\T^*\1_t(s) = T(t,s), \quad t\in[0,1].
$$ 
The domain of $\T^*$ is the Hilbert space $\dom(\T^*)$ that is generated by the indicators $\1_t$, $t\in[0,1]$, and closed under the inner product
$$
{\la \1_t, \1_s \ra}_{\dom(\T^*)} = \int_0^1 T(t,u)T(s,u)\, \d u.
$$
\end{dfn}

We note that $\dom(\T^*)$ may not be a function space in general. Also, we note that, by construction, $\T^*$ is an isometry to a subspace of $\L^2$ and
$$
{\la f, g \ra}_{\dom(\T^*)}
=
{\la \T^* f, \T^* g\ra }_{\L^2}. 
$$

The following result gives more understanding on the operator $\T^*$ by considering the case where the kernel $T$ is continuous and of bounded variation in its first argument, and $T(0,\cdot)\equiv 0$.  In relation to Fredholm Gaussian processes $X$, the condition $T(0,\cdot)\equiv 0$ simply means that $X_0\equiv 0$, if $T$ is the Fredholm kernel of $X$.

\begin{lmm}\label{lmm:bv}
If the kernel $T$ is left-continuous and of bounded variation in its first argument and $T(0,\cdot)\equiv 0$, then 
\begin{equation}\label{eq:bv}
\T^*f(t) = \int_0^1 f(s)\, T(\d s, t),
\end{equation}
for all $f\in\mathcal{E}$, and $\T^*$ is the adjoint of $\T$ in the sense that 
\begin{equation}\label{eq:adjoint}
\int_0^1 \T^*f(t)\,  g(t)\, \d t 
=
\int_0^1 f(t)\,  \T g(\d t).
\end{equation}
for all $f\in\mathcal{E}$ and $g\in\L^2$.
\end{lmm}

\begin{proof}
For \eqref{eq:bv} it is enough to show that its right-hand-side is $T(t,s)$ for $f(s)=\1_t(s)$. But  this is straightforward:
\begin{eqnarray*}
T(t,s) &= T(t-,s) - T(0,s) 
&=
\int_0^t T(\d u, s) \\
&=
\int_0^1 \1_t(u)\, T(\d u, s).
\end{eqnarray*}

Let us then show \eqref{eq:adjoint}. By the Fubini theorem
\begin{eqnarray*}
\int_0^1 \T^*f(t)\,  g(t)\, \d t 
&=&
\int_0^1\int_0^1 f(s)T(\d s, t)\, g(t)\, \d t \\
&=&
\int_0^1 f(t)\int_0^1T(\d t, s)\,g(s)\, \d s  \\
&=&
\int_0^1 f(t)\,  \T g(\d t),
\end{eqnarray*}
which proves the claim.
\end{proof}

\begin{exa}\label{exa:indicator_adjoint}
For the integral operator $\I$ we have
$$
\I^* f(t) = \int_0^1 f(s)\,\delta_t(\d s),
$$
where $\delta_t$ is the unit mass at point $t$. In other words, $\I^*$ is the identity operator. This also provides an example where \eqref{eq:bv} holds provided that $T(\d s,t)$ is understood as a measure. 
\end{exa}
\begin{rem}{\rm
If the kernel $T$ is Volterra type (i.e. $T(t,s)=0$ if $s>t$) and of bounded variation in its first argument, then the operator $\T^*$ coincides with the adjoint operator of Lemma 1 of Al\`os at al. \cite{Alos-Mazet-Nualart-2001}. 
}
\end{rem}
\subsection{Gaussian Fredholm Processes}\label{fredholm}

Let $(\Omega,\F,\P)$ be a probability space.  For concreteness, we assume that $\Omega=\L^2$, $\F$ is the associated Borel $\sigma$-field and $\P$ is the probability measure of the co-ordinate process $X_t(\omega)=\omega(t)$. The following result can be found from \cite{Sottinen-Viitasaari-2014-preprint} although here we present the statement in a slightly different form. 

\begin{lmm}[Fredholm Representation]\label{lmm:fredholm}
Suppose $X=(X_t)_{t\in[0,1]}$ is a centered process with covariance $R$. Then $X$ takes values in $\L^2$ if the covariance operator $\R$ is trace-class, i.e.,
\begin{equation}\label{eq:trace-assumption}
\int_0^1 R(t,t)\, \d t < \infty.
\end{equation}
In this case the square root $\K$ of the covariance $\R$ admits a \emph{Fredholm kernel}, i.e.,
$$
R(t,s) = \int_0^1 K(t,u) K(s,u)\, \d u.
$$
The kernel $K\in \L^2\times\L^2$ can be assumed to be positive symmetric, and in this case it is unique. Consequently, we have the \emph{Fredholm representation} for Gaussian processes with trace-class covariance operator $\R$:
\begin{equation}\label{eq:fredholm}
X_t = \int_0^1 K(t,s)\, \d W_s
\end{equation}
for some Brownian motion.  
\end{lmm}

\begin{rem}[Construction of Gaussian Fredholm Processes]
Lemma \ref{lmm:fredholm} can be used to construct a Gaussian process having values in $\L^2$.  Indeed, let $(e_n)_{n=1}^\infty$ be any orthonormal basis on $\L^2$, and let $(\xi_n)_{n=1}^\infty$ be i.i.d. standard Gaussian random variables.  Then a Gaussian process $X$ with Fredholm kernel $K$ can be constructed using the following $\L^2$-convergent series:
$$
X_t = \sum_{n=1}^\infty \int_0^1 \e_n(s) K(t,s)\, \d s \,\, \xi_n.
$$ 
\end{rem}

\begin{dfn}[Cameron--Martin Space]\label{dfn:cm}
The \emph{Cameron--Martin space}, or the \emph{Reproducing kernel Hilbert space}, $\CM$ of a centered Gaussian process $X$ with covariance $R$ is the Hilbert space of real-valued functions on $[0,1]$ generated by the functions $R(t,\cdot)$, $t\in[0,1]$, and the inner product
$$
{\la R(t,\cdot), R(s,\cdot) \ra}_{\CM} = R(t,s).
$$
\end{dfn}

For Gaussian Fredholm processes with representation \eqref{eq:fredholm} the implicit Definition \ref{dfn:cm} can be made completely concrete. Indeed, in this case we have $\CM = \K\L^2$ and the inner product is given by
$$
{\la f, g \ra}_{\CM} = \int_0^1 \K^{-1}f(t)\, \K^{-1}g(t)\, \d t.
$$

\begin{exa}
The Brownian motion is a Gaussian Fredholm process with Fredholm kernel $K(t,s)=\1_t(s)$. Consequently, the Cameron--Martin space of a Brownian motion is $\I\L^2$ and the inner product is
$$
{\la f, g\ra}_{\CM} = \int_0^1 f'(t) g'(t)\, \d t.
$$
\end{exa}

\begin{dfn}[Linear Space]
The \emph{linear space} $\H_1$ or the \emph{first chaos} of a centered Gaussian process $X$ is the closed subspace of $\L^2(\Omega,\sigma(X),\P)$ generated by the random variables $X_u$, $u\in[0,1]$.
\end{dfn}

\begin{dfn}[Integrand Space]\label{dfn:integrand-space}
The \emph{integrand space} $\mathcal{I}$ of a centered Gaussian process $X$ with covariance $R$ is the closure of step-functions $f\in\mathcal{E}$ under the norm induced by the inner product generated by the relation 
$$
{\la \1_t,\1_s\ra}_{\mathcal{I}} = R(t,s).
$$
\end{dfn}

In general, the Hilbert space $\mathcal{I}$ may contain distributions. Note also that $\mathcal{I} = \dom(\K^*)$.

Suppose the centered Gaussian process $X\colon\Omega\to\L^2$ is infinite-dimensional in the sense that its Cameron--Martin space $\CM$ is infinite-dimensional.  Then all the spaces $\CM$, $\H_1$ and $\mathcal{I}$ are isometric to $\L^2$. For example, $\K\colon\L^2\to\CM$ is an isometry and $\K^*\colon\mathcal{I}\to\L^2$ is an isometry.  

\begin{dfn}[Abstract Wiener Integral]\label{dfn:awi}
Let $X$ be a Gaussian Fredholm process with kernel $K$. Let $f\in\mathcal{I}$. The \emph{abstract Wiener integral} 
$$
\int_0^1 f(t)\, \d X_t
$$
is the image in $\H_1$ under the isometry built from linearly extending the mapping $\1_u \mapsto X_u$.
\end{dfn}

Finally, we note that the Fredholm representation \eqref{eq:fredholm} extends into the following \emph{transfer principle} (see \cite{Sottinen-Viitasaari-2014-preprint} for details).  

\begin{lmm}[Transfer Principle]
Let $X$ be a Gaussian Fredholm process with kernel $K$ and associated Brownian motion $W$. Then for all $f\in\mathcal{I}$  
$$
\int_0^1 f(t)\, \d X_t = \int_0^1 \K^*f(t)\, \d W_t,
$$
where the left-hand-side is an abstract Wiener integral and the right-hand-side is the classical Wiener integral.
\end{lmm}

Finally, we note that by taking $(e_n)_{n=1}^{\infty}$ to be an orthonormal basis of $\L^2$, one can construct orthonormal bases for $\CM$, $\H_1$ and $\mathcal{I}$.  For example, $(\K e_n)_{n=1}^\infty$ is an orthonormal basis on $\CM$.


\subsection{Classical Malliavin Differentiation}\label{sec:classical-malliavin}
We recall briefly the essential elements of Malliavin calculus. For further details, see Nualart \cite{Nualart-2006}, Nualart and Nualart \cite{N-N-MalliavinCalculus}, and Nourdin and Peccati \cite{Nourdin-Peccati-2012}.

\subsubsection{Isonormal Processes and Wiener-It\^{o} Chaos Expansion}\label{ss:isonormal}
Let $ \HH$ be a real separable Hilbert space. 
For any $q\in\mathbb{N}$, we denote by $ \HH^{\otimes q}$ and $ \HH^{\odot q}$, respectively, the $q$th tensor power and the $q$th symmetric tensor power of $ \HH$. We also set by convention
$ \HH^{\otimes 0} =  \HH^{\odot 0} =\mathbb{R}$. 

\begin{rem}
If $\HH = \L^2(A,\mathcal{A}, \mu) =\L^2(\mu)$, where $\mu$ is a $\sigma$-finite
and non-atomic measure on the measurable space $(A,\mathcal{A})$, then $ \HH^{\otimes q} = \L^2(A^q,\mathcal{A}^q,\mu^q)=L^2(\mu^q)$, and $ \HH^{\odot q} = \L_s^2(A^q,\mathcal{A}^q,\mu^q) = \L_s^2(\mu^q)$, 
where $\L_s^2(\mu^q)$ stands for the subspace of $\L^2(\mu^q)$ composed of those functions that are $\mu^q$-almost everywhere symmetric. 
\end{rem}

We denote by $W=\{W(h) \,;\, h\in  \HH\}$
the {\it isonormal Gaussian process} over $ \HH$. This means that $W$ is a centered Gaussian family, defined on some probability space $(\Omega ,\mathcal{F},\P)$, with a covariance structure given by the relation
$\E\left[ W(h)W(g)\right] ={\langle h,g\rangle}_{ \HH}$. We also assume that $\mathcal{F}=\sigma(W)$, that is, $\mathcal{F}$ is generated by $W$, and use the shorthand notation $\L^2(\Omega) = \L^2(\Omega, \mathcal{F}, \P)$.

\begin{rem}
The isonormal Gaussian process can be constructed from a centered Gaussian process as follows.  Let $X$ be a centered Gaussian process with covariance $R$ and associated integrand space $\mathcal{I}$.  Then $W(h) = \int_0^1 h(t)\, \d X_t$, $h\in\mathcal{I}$, is the isonormal Gaussian process with $\mathcal{H}=\mathcal{I}$. In particular, the Brownian motion corresponds to the isonormal Gaussian process with $\HH=\L^2$.
\end{rem}

For $q\geq 1$, let $\mathcal{H}_{q}$ be the {\it $q$th chaos} of $W$, defined as the closed linear subspace of $\L^2(\Omega)$
generated by the family $\{H_{q}(W(h))\,;\, h\in  \HH,\left\| h\right\| _{ \HH}=1\}$, where $H_{q}$ is the $q$th Hermite polynomial, defined as
$$
H_q(x) = (-1)^q \e^{\frac{x^2}{2}}
\frac{\d^q}{\d x^q} \left[ \e^{-\frac{x^2}{2}} \right].
$$
We write by convention $\mathcal{H}_{0} = \mathbb{R}$. The mapping 
$$
\mathbf{I}_{q}(h^{\otimes q})=H_{q}(W(h))
$$ 
can be extended to a
linear isometry between the symmetric tensor product $ \HH^{\odot q}$
(equipped with the modified norm $\sqrt{q!}{\left\| \,\cdot\, \right\|}_{ \HH^{\otimes q}}$)
and the $q$th Wiener chaos $\mathcal{H}_{q}$. For $q=0$, we write by convention $\mathbf{I}_{0}(c)=c$, $c\in\mathbb{R}$.

It is well-known that $\L^2(\Omega)$ can be decomposed into the infinite orthogonal sum of the spaces $\mathcal{H}_{q}$: any square-integrable random variable
$F\in\L^2(\Omega)$ admits the following {\it Wiener-It\^{o} chaotic expansion}
$$
F=\sum_{q=0}^{\infty }\mathbf{J}_q(F),  \label{E}
$$
where the series converges in $\mathcal{L}^2(\Omega)$ and $\mathbf{J}_{q}$ is the
orthogonal projection operator on the $q$th chaos $\mathcal{H}_q$. 

\subsubsection{Malliavin Operators}\label{ss:mall}

We  briefly introduce some basic elements of the Malliavin calculus with respect
to the isonormal Gaussian process $W$. 

Let $\mathcal{S}$ be the set of all
cylindrical random variables of the form
$$
F=g\left( W(\phi _{1}),\ldots ,W(\phi _{n})\right) ,  \label{v3}
$$
where $n\geq 1$, $g\colon\mathbb{R}^{n}\rightarrow \mathbb{R}$ is an infinitely
differentiable function such that it and all its partial derivatives have at most polynomial growth, and $\phi _{i}\in  \HH$,
$i=1,\ldots,n$.
The {\it Malliavin derivative}  of $F$ with respect to $W$ is the element of $\L^2(\Omega ;\HH)$ defined as
\begin{equation*}
\bD F\;=\;\sum_{i=1}^{n}\frac{\partial g}{\partial x_{i}}\left( W(\phi_{1}),\ldots ,W(\phi _{n})\right) \phi _{i}.
\end{equation*}
In particular, $\bD W(h)=h$ for every $h\in  \HH$. 
By iteration, one can define the $m$th order derivative $\bD ^{m}F$, which is an element of $\L^2(\Omega ; \HH^{\odot m})$
for every $m\geq 2$.
For $m\geq 1$ and $p\geq 1$, let ${\mathbb{D}}^{m,p}$ denote the closure of
$\mathcal{S}$ with respect to the norm ${\Vert \cdot \Vert}_{m,p}$, defined by
the relation
\begin{equation*}
{\Vert F\Vert}_{m,p}^{p}\;=\;\E\left[ {|F|}^{p}\right] +\sum_{i=1}^{m}\E\left[
{\Vert \bD ^{i}F\Vert}_{ \HH^{\otimes i}}^{p}\right].
\end{equation*}
By Proposition 1.2.1 of Nualart \cite{Nualart-2006} and the following discussion there, the (iterative) Malliavin derivatives $\bD^m$ are closable, and can thus be extended to the spaces $\mathbb{D}^{m,p}$ for any $p\ge 1$.

The Malliavin derivative $\bD$ obeys the following chain rule: If
$\varphi\colon\mathbb{R}^{n}\rightarrow \mathbb{R}$ is continuously
differentiable with bounded partial derivatives and if $F=(F_{1},\ldots
,F_{n})$ is a vector of elements in ${\mathbb{D}}^{1,2}$, then $\varphi
(F)\in {\mathbb{D}}^{1,2}$ and
$$
\bD\,\varphi (F)=\sum_{i=1}^{n}\frac{\partial \varphi }{\partial x_{i}}(F)\bD F_{i}.
$$

The operator $\mathrm{L}$, defined as 
$$
\mathrm{L}=\sum_{q=0}^{\infty }-qJ_{q},
$$
is the {\it infinitesimal generator of the Ornstein-Uhlenbeck semigroup}. The domain of $\mathrm{L}$ is
\begin{equation*}
\dom(\mathrm{L})=\left\{F\in L^2(\Omega )\,;\,\sum_{q=1}^{\infty }q^{2}{\left\|
J_{q}F\right\|}_{\L^2(\Omega )}^{2}<\infty \right\}=\mathbb{D}^{2,2}.
\end{equation*}
For $F \in\L^2(\Omega )$, we define 
$$
\mathrm{L}^{-1}F =\sum_{q=1}^{\infty }-\frac{1}{q} J_{q}(F). 
$$
The operator $\mathrm{L}^{-1}$ is called the
\emph{pseudo-inverse} of $\mathrm{L}$. Indeed, for any $F \in\L^2(\Omega )$, we have $\mathrm{L}^{-1} F \in  \dom{\mathrm{L}} = \mathbb{D}^{2,2}$,
and
$$
\mathrm{L}\mathrm{L}^{-1} F = F - \E[F].
$$

The following integration-by-parts formula can be found e.g. in Nourdin and Peccati \cite{Nourdin-Peccati-2012}, Theorem 2.9.1.

\begin{pro}\label{pro : Tech1}
Suppose that $F\in\mathbb{D}^{1,2}$ and $G\in L^2(\Omega)$. Then, $\mathrm{L}^{-1}G \in \mathbb{D}^{2,2}$ and
$$
\E[FG] = \E[F]\E[G]+\E\left[{\langle \bD F,-\bD \mathrm{L}^{-1}G\rangle}_{\HH}\right].
$$
\end{pro}

\section{Pathwise Differentiation}
\label{sec:differentiation}
We introduce several (classical) pathwise derivatives as well as a pathwise version of the Malliavin differentiation in the path space $\L^2$ without a priori assuming an isonormal Gaussian structure as explained in Section \ref{sec:classical-malliavin}.

\begin{dfn}
Let $f\colon\L^2\to\mathbb{C}$.
\begin{description}
\item[(a)] 	\textbf{(Fr\'echet derivative)}
The \emph{Fr\'echet derivative} of $f$ at point $x\in\L^2$ is the element $\nabla f(x)\in\L^2$ such that for all $y\in\L^2$
$$
\lim_{{\|y\|}_{\L^2}\to 0}\frac{\left|f(x+y) - f(x) - {\la \nabla f(x), y \ra}_{\L^2}\right|}{{\|y\|}_{\L^2}}
=
0.
$$
\item[(b)] \textbf{(G\^ateaux derivative)}
The \emph{G\^ateaux derivative} of $f$ at point $x\in\L^2$ to direction $y\in\L^2$ is
$$
\nabla_y f(x) = \lim_{\varepsilon\to 0} \frac{f(x+\varepsilon y)-f(x)}{\varepsilon}.
$$
\item[(c)] \textbf{(Pathwise Malliavin derivative)} Let $\mathcal{C}^{\infty}_p (\mathbb{R}^n)$ denote the space of all polynomially bounded functions with polynomially bounded partial derivatives of all orders.  Consider functionals $f\colon\L^2\to\mathbb{C}$ of the form
$$
f(x) = g \left(z_1,\ldots, z_n\right),
$$
where $n \in \mathbb{N}$ and $g \in \mathcal{C}^{\infty}_p (\mathbb{R}^n)$, and 
\begin{equation}\label{eq:smooth-components}
z_k = \int_0^1 e_k(t) \, \d x(t)
\end{equation}
for some elementary functions $e_k \in \mathcal{E}$. For such $f$ we write $f\in\cS$. We call the elements of class $\cS$ the {\it smooth} functionals. The \emph{pathwise Malliavin derivative} of such $f\in\cS$ is
\begin{equation}\label{eq:MD-smooth}
\D_t f(x) = \sum_{k=1}^n \frac{\partial}{\partial z_k} g(z_1,\ldots,z_n)\,e_k(t).
\end{equation}
More generally, by iteration for every $m \in \mathbb{N}$, the pathwise Malliavin derivative of order $m$ is defined as follows: for every $t_1,...,t_m \in [0,1]$,
\begin{eqnarray*}
\lefteqn{\D^m_{t_m,...,t_1} f(x)} \\
&=& \sum_{1\le k_1,...,k_m \le n } \frac{\partial ^m}{\partial z_{k_1} \cdots \partial z_{k_m}} g (z_{k_1}, ..., z_{k_n})
\left( e_{k_1}\otimes \cdots \otimes e_{k_m}\right) (t_1,...,t_m).
\end{eqnarray*}
\end{description}
\end{dfn}
\begin{rem}
If $f$ is Fr\'echet differentiable at point $x \in \L^2$, then the G\^ateaux derivative can be written as 
$$
\nabla_y f(x) = {\la \nabla f(x),  y\ra}_{\L^2}
= \int_0^1 \nabla_t f(x)\, y(t)\, \d t. 
$$
Throughout the article the notation $\nabla$ can mean either the Fr\'echet differential or the G\^ateaux derivative, whenever confusion cannot arise.
\end{rem}
\begin{rem}
If $X$ is a Gaussian process, then our definition of the pathwise Malliavin derivative coincides with the classical one introduced in Section \ref{ss:mall} on the class of smooth functionals. In particular, it does not depend on the representation (\ref{eq:smooth-components}). For details in this case, we refer to Nualart \cite{Nualart-2006}.  
\end{rem}

\begin{rem}[Caution] 
\begin{itemize}
\item[(a)] 
In our definition of pathwise Malliavin derivative, it is not necessary to take $\L^2$ as the domain of smooth functionals.  In fact, any suitable space of functions can be realized as good integrators with respect to elementary functions. However, the $\L^2$ space can be seen as a convenient reference function space later on, since as we are going to apply our results in a setting where $x$ in \eqref{eq:smooth-components} plays the role of a typical sample path over the interval $[0,1]$ of a Gaussian process.   
\item[(b)] 
The $\mathcal{E}$-valued operator $\D$ defined in \eqref{eq:MD-smooth} is in fact a linear unbounded operator.  It is well known (cf. Nualart \cite{Nualart-2006}, Lemma 1.2.1) that the domain of the classical Malliavin derivative can be extended in $\L^2(\Omega)$ fashion, if $\P=\P^X$ is a Gaussian measure. The key part in the extension is the Gaussian integration-by-parts formula  
$$
\E\left[{\la \D F,h\ra}_\HH \right]
=
\E\left[FW(h) \right],
$$
$h\in\HH$, $h\in\cS$,
which implies the closability of the Malliavin derivative as an operator $\D\colon\L^2(\Omega)\to \L^2(\Omega;\mathcal{I})$, as shown in Nualart \cite{Nualart-2006}, Proposition 1.2.1. However, in the pathwise setting the closability of operator $\D$ is not available. On the other hand, surprisingly such requirement is not needed in order to establish our results.
\end{itemize}
\end{rem}

The next lemma relates the pathwise Malliavin derivative to the Fr\'echet derivative. A similar result can be found in Nualart and Saussereau \cite{NuaSau} for the particular case of the fractional Brownian motion.  
\begin{lmm}\label{lmm:frechet-malliavin}
Let $f\in\cS$ and $y\in\L^2$.  Then
$$
{\la \nabla f(x), \I y \ra}_{\L^2}
=
{\la \D f(x), y\ra}_{\L^2}
$$
\end{lmm}

\begin{proof}
	Straightforward calculations yield, with $e_{k} \in \mathcal{E}$ and $x\in \L ^ {2}$, that
	\begin{eqnarray*}
		\lefteqn{f(x+\I y) } \\
		&=&
		g\left(\int_0^1 e_1(t)\, \d(x(t)+\I y(t))\,,\ldots,\, \int_0^1 e_n(t)\,\d (x(t)+ y(t))\right) \\
		&=&
		g\left(\int_0^1 e_1(t)\,\d x(t)+ \int_0^1 e_1(t)\,\d\I y(t)\,,\ldots,\, \int_0^1 e_n(t)\,\d x(t)+ \int_0^1 e_n(t)\,\d\I y(t)\right) \\
		&=&
		g\left(\int_0^1 e_1(t)\,\d x(t)+ \int_0^1 e_1(t)y(t)\,\d t\,,\ldots,\, \int_0^1 e_n(t)\,\d x(t)+ \int_0^1 e_n(t)y(t)\, \d t\right) \\
		&=&
		f \left(z_1 + {\la y_1,y\ra}_{\L^2}\,,\ldots,\, z_1 + {\la y_n,y\ra}_{\L^2}\right).
	\end{eqnarray*}
	Thus,
	\begin{eqnarray*}
		\lefteqn{{\la \nabla f(x), \I y \ra}_{\L^2}
		=
		\nabla_{\I y} f(x)} \\
		&=&
		\lim_{\varepsilon\to 0} \frac{f(x+\varepsilon \I y) - f(x)}{\varepsilon} \\
		&=&
		\lim_{\varepsilon\to 0} \frac{f(z_1 + \varepsilon{\la y_1,y\ra}_{\L^2}\,,\ldots,\, z_1 + \varepsilon{\la y_n,y\ra}_{\L^2}) - f(z_1,\ldots,z_n)}{\varepsilon} \\
		&=&
		\sum_{k=1}^n\frac{\partial}{\partial z_k}f(z_1,\ldots,z_n)\, {\la y_k,y\ra}_{\L^2} \\
		&=&
		{\la \D f(x), y \ra}_{\L^2}
	\end{eqnarray*}
proving the claim.
\end{proof}

We also need the following two lemmas in order to establish our novel integration-by-parts formulas in Section \ref{sec:IBP}.
\begin{lmm}\label{lmm:D=D}
Let $X=(X_t)_{t \in [0,1]}$ be a centered Gaussian process satisfying \eqref{eq:trace-assumption}, and let $\bD$ denote 
 the associated classical Malliavin derivative. Then for every $f \in \cS$ we have $f(X) \in \dom (\bD)$ and that $\D_t f (X) = \bD_t f (X)$  for every $t \in [0,1]$. More generally, for every $m \in \mathbb{N}$, and $t_1,...,t_m \in [0,1]$, it holds that, $\D_{t_m,...,t_1} f(X) \in \dom (\bD)$, and 
$$
\D_t \left(\D_{t_m,...,t_1} f(X) \right) = \D_{t,t_m,...,t_1} f(X) = \bD_{t,t_m,...,t_1} f(X)
$$
for every $t \in [0,1]$. 
\end{lmm}
\begin{proof}
By Lemma \ref{lmm:fredholm} the paths of the process $X$ belong to $\L^2$. Moreover,  $\mathcal{E} \subset \mathcal{I}$, where $\mathcal{I}$ stands for the associated integrand space of Gaussian process $X$. Furthermore, for every $e \in \mathcal{E}$, the pathwise integral $ \int_{0}^{1} e(t) \d X_t$ coincides with the abstract Wiener integral of Definition \ref{dfn:awi}. Thus, the claim follows. 
\end{proof}

\begin{lmm}\label{lmm:Domain-K*}
Let $f \in \cS$, $K \in \L^2 \times \L^2$, and let $X$ have paths in $\mathcal{L}^2$. Then $f(X)$ is twice pathwise Malliavin differentiable and the mapping $\D^2_{t,\,\cdot\,} f(X)$ belongs to $\dom (\K^*)$.
\end{lmm}

\begin{proof}
Second order pathwise Malliavin differentiability is obvious. Indeed, we have
$$
\D^2_{t,\,\cdot\,} f(X) = \sum_{1 \le k,l\le n} \frac{\partial^2}{\partial z_k\partial z_l} g\left( \int_{0}^{1} e_1 (t) \d X_t, ..., \int_{0}^{1} e_n (t) \d X_t\right)  \left( e_k \otimes e_l\right) (t,\cdot).
$$ 
This also shows that $\D^2_{t,\cdot} f(X)\in\dom(\K^*)$, since $\mathcal{E}\subset\dom(\K^*)$.   
\end{proof}

\section{Integration-by-Parts Characterization of Gaussian Processes}
\label{sec:IBP}
We begin with the following stronger formulation of the integration-by-parts characterization.
\begin{thm}[General Gaussian Processes, Strong Version]\label{thm:ibp-gfp}
The co-ordinate process $X\colon\Omega\to\L^2$ is centered Gaussian with Fredholm kernel $K\in\mathcal{L}^2\times\mathcal{L}^2$ if and only if 
\begin{equation}\label{eq:ibp-gfp-strong}
\E\left[ X_t  \D_t f(X)\right]
=
\E\left[\int_0^1 K(t,s)\K^*\left[\D^2_{t,\,\cdot\,} f(X)\right](s)\, \d s\right]
\end{equation}
for all $t\in[0,1]$ and $f\in\cS$. 
\end{thm}

\begin{rem}
If $X_0=0$ and the kernel $K$ is left-continuous and of bounded variation in its first argument, then we can 
reformulate \eqref{eq:ibp-gfp-strong} as
$$
\E\left[X_t \D_t f(X)\,\right]
=
\E\left[\int_0^1\!\!\!\int_0^1 K(t,s)\D^2_{t,u} f(X)\, K(\d u,s)\,\d s\right].
$$
\end{rem}

\begin{proof}[Proof of Theorem \ref{thm:ibp-gfp}] 

\emph{``If'' part:}\quad
Suppose the co-ordinate process $X\colon\Omega\to\L^2$ satisfies \eqref{eq:ibp-gfp-strong}. We begin by considering the covariance function of $X$, which will justify the use of the Fubini theorem later and make a tedious variance calculations unnecessary.  For this, take $ f(X) = \frac12 X_u^2$ for some $u\in[0,1]$. Then $f \in \cS$. We have $\D_t f(X) = X_u\1_u(t)$ and $\D_{t,s} f(X) = \1_u(s)\1_u(t)$.
Consequently, \eqref{eq:ibp-gfp-strong} yields
\begin{eqnarray}\nonumber
\E\left[ X_t X_u\right] \1_u(t)
&=&
\E\left[\int_0^1 K(t,s)\K^*\left[\1_u\right](s)\, \d s\right] \1_u(t) \\ \nonumber
&=&
\int_0^1 K(t,s) K(u,s)\, \d s\, \1_u(t) \\ \label{eq:cov-comp}
&=&
R(t,u)\1_u(t).
\end{eqnarray}
This shows that $X$ has the covariance function $R$ given by the Fredholm kernel $K$. In particular, we have 
$$
\E\left[X_t^2\right] = \int_0^1 K(t,s)^2\, \d s,
$$
and since $K \in\L^2\times\L^2$, we have $\int_0^1 \E\left[X_t^2\right] \, \d t < \infty$
which justifies the use of the Fubini theorem in the rest of the proof. Next we are going to show that any finite linear combination 
$$
Z = \sum_{k=1}^{n} a_k \left(  X_{t_k} - X_{t_{k-1}}\right) = \int_{0}^{1} e(t) \, \d X_t
$$ 
with $e= \sum_{k}^{n} a_k \textbf{1}_{(t_{k-1},t_k]} \in \mathcal{E}$ is a Gaussian random variable. Now, note that for every $\theta$ the complex-valued exponential functional $e^{\i \theta Z} = \cos ( \theta Z) + \i \sin (\theta Z)$ belongs to $\cS$, meaning that the real and imaginary parts both belong to $\cS$. 
Let $\varphi$ be the characteristic function of $Z$. Then
\begin{eqnarray*}
	\D_t \e^{\i\theta Z} &=& \i\theta e(t) \,\e^{\i\theta Z}, \\
	\D^2_{t,s} \e^{\i\theta Z} &=& -\theta^2 e(t)e(s)\, \e^{\i\theta Z}. 
\end{eqnarray*}
Hence $	\E\left[X_t\D_t \e^{\i\theta Z}\right] =	\i\theta e(t) \E\left[X_t \e^{\i\theta Z}\right]$. Also, by a direct application of Fubini theorem 
\begin{eqnarray*}
\lefteqn{\E\left[\int_0^1 K(t,s)\K^*\left[\D^2_{t,\,\cdot\,}\e^{\i\theta Z}\right](s)\, \d s \right]}  \nonumber \\
&=&
-\E\left[\int_0^1 K(t,s)\K^*\left[\theta^2 e(t)e(\cdot)\e^{\i\theta Z}\right](s)\, \d s \right] \nonumber \\
&=&	-\theta^2 e(t)\E\left[\int_0^1 K(t,s)\K^*\left[e(\cdot)\e^{\i\theta Z}\right](s)\, \d s \right] \nonumber \\
&=&	-\theta^2 e(t)\int_0^1 K(t,s)e^*(s)\, \d s\,\, \E\left[\e^{\i\theta Z}\right], 
\end{eqnarray*}
where we have denoted $e^* = \K^* e$. Consequently, the integration-by-parts formula \eqref{eq:ibp-gfp-strong} yields
\begin{equation}\label{eq:pre-diff}
\i\, \E\left[ X_t \e^{i\theta Z} \right]
=
-\theta \int_0^1 K(t,s)e^*(s)\, \d s \, \, \varphi(\theta).
\end{equation}
By Fubini theorem justified by the covariance computation \eqref{eq:cov-comp}, we also have
$$
\varphi'(\theta) = \E\left[\i Z \, \e^{\i\theta Z}\right].
$$
Thus we obtain by several application of \eqref{eq:pre-diff} that $\varphi'(\theta) = - c \theta \, \varphi(\theta)$,
where we have denoted 
$$
c = \int_0^1 \left( \sum_{k=1}^{n} a_k \left( K(t_k,s) - K(t_{k-1},s) \right) e^*(s) \right)  \,\d s < \infty.
$$
This implies that $\varphi(\theta) = \e^{-\frac{1}{2} c \theta^2}$, and since $\varphi$ is a characteristic function, $c>0$. Consequently, $Z$ is a centered Gaussian random variable with variance $c$. 

\emph{``Only if'' part:} Since the co-ordinate process $X\colon \Omega\to\mathcal{L}^2$ is Gaussian, we have the full power of Malliavin calculus at our disposal.
In particular, we can use Proposition \ref{pro : Tech1} with $F=\D_t f(X)$ and $G=X_t$. Since $\E[X_t]=0$, we obtain
$$
\E\left[X_t\D_tf(X)\right] =
\E\left[{\la \D^2_{t,\cdot} f(X), -\D\mathrm{L}^{-1} X_t \ra}_{\mathcal{I}}\right]
$$
But $-\D\mathrm{L}^{-1} X_t = \1_t$ and $\K^*$ is an isometry between $\mathcal{I}$ and $\mathcal{L}^2$.  Therefore, by noticing that $\K^*\1_t(s) = K(t,s)$, we obtain
\begin{eqnarray*}
\E\left[X_t\D_tf(X)\right] 
&=&
\E\left[{\la \D^2_{t,\cdot} f(X), \1_t \ra}_{\mathcal{I}}\right] \\
&=&
\E\left[{\la \K^*\D^2_{t,\cdot} f(X), \K^*\1_t \ra}_{\mathcal{L}^2}\right] \\
&=&
\E\left[\int_0^1\K^*\left[\D^2_{t,\cdot} f(X)\right](s) \K^*\1_t(s) \, \d s\right] \\
&=&
\E\left[\int_0^1\K^*\left[\D^2_{t,\cdot} f(X)\right](s) K(t,s) \, \d s\right]
\end{eqnarray*}
showing the claim.
\end{proof}

\begin{rem}
It is classical that a random variable $X \approx \mathscr{N}(0,\sigma^2)$ if and only if its characteristic function $\varphi_{X}$  satisfies $\varphi'_{X} (\theta) \approx - \sigma^2 \theta \varphi_{X}(\theta)$. The latter is equivalent to 
\begin{equation}\label{eq:cf-ode}
\E \left[   X e^{\i \theta X} \right] \approx   \i \,\sigma^2 \,  \theta \,  \E \left[   e^{\i \theta X}\right].
\end{equation}
Hence, as a direct consequence, relation \eqref{eq:cf-ode} is also equivalent to the fact that for a given diffusive (satisfying the chain rule) gradient operator $D$ on the space of random variables, it holds that  
$$
\E \left[ X D e^{\i\theta X}  \right] \approx \sigma^2  \, \E \left[ D^2 e^{\i\theta X}\right].  
$$
For example, in the setting of Theorem \ref{thm:ibp-gfp} and for a random variable $X_t$, by considering the functional $f(X) = e^{\i \theta X_t}$, one can easily infer that 
$$
\E \left[  X_t e^{\i \theta X_t }\right] = \i \, \theta \, \int_{0}^{1} K(t,s) ^2 \,\d s \, \E \left[ e^{\i \theta X_t} \right].
$$ 
This implies that $X_t \sim \mathscr{N}(0, \sigma^2)$ with $\sigma^2 = \int_{0}^{1} K(t,s) ^2 \d s$. Indeed, Theorem \ref{thm:ibp-gfp} is a functional version of the aforementioned considerations in order to capture the Gaussian structure of $X$ as a process.   
\end{rem}

If we have additional information on the co-ordinate process, then we can obtain a weaker integration-by-parts characterization. This is the topic of the next theorem.

\begin{thm}[General Gaussian Processes, Weak Version]\label{thm:ibp-gfp-weak}
	Let $K\in\L^2\times\L^2$ be a square integrable kernel. Assume that the co-ordinate process $X\colon\Omega\to\L^2$ satisfies $X \in \L^2(\d t \otimes \P)$, i.e.
	\begin{equation}\label{eq:strong_assumption}
	\int_0^1 \E \left[X_t^2\right] \d t < \infty.
	\end{equation}
	Then $X$ is centered Gaussian with the Fredholm kernel $K$ if and only if
	\begin{equation}\label{eq:ibp-gfp}
	\E\left[\int_0^1 X_t \D_t f(X)\, \d t\right]
	=
	\E\left[\int_0^1\!\!\!\int_0^1 K(t,s)\K^*\left[\D^2_{t,\,\cdot\,}f(X)\right](s)\, \d s\d t\right]
	\end{equation} 
	for all $f\in\cS$. 
\end{thm}

Before proving Theorem \ref{thm:ibp-gfp-weak}, let us consider its similarities and differences to the integration-by-parts characterization of finite-dimensional Gaussian vectors. 

\begin{rem}
	A random vector $X=(X_1,\ldots, X_d)$ is centered Gaussian with covariance matrix $\R$ if and only if
	\[
	\E\left[\sum_{i=1}^d X_i \frac{\partial}{\partial x_i}f(X)\right]
	=
	\E\left[\sum_{i=1}^d\sum_{j=1}^d \R_{ij}\frac{\partial^2}{\partial x_i\partial x_j}f(X)\right]
	\]
for all smooth$f\colon\mathbb{R}^d\to\mathbb{R}$ such that the expectations above exist.
	Thus, by a simple analogy, one would guess (wrongly!) that a process is Gaussian if and only if
	\[
	\E\left[\int_0^1 X_t\D_t f(X)\, \d t\right]
	=
	\E\left[\int_0^1\!\!\!\int_0^1 R(t,s)\D^2_{t,s}f(X)\, \d s \d t\right].
	\]
	This formula is not, however, true even for the  Brownian motion.  It seems that there is no integration-by-parts formula in terms of the covariance directly and the simplest formula one can obtain is \eqref{eq:ibp-gfp} that is given in terms of the Fredholm kernel.  
\end{rem}

\begin{proof}[Proof of Theorem \ref{thm:ibp-gfp-weak}]
By \eqref{eq:strong_assumption} and the Fubini theorem, the weak integration-by-parts formula \eqref{eq:ibp-gfp} follows from the strong integration-by-parts formula \eqref{eq:ibp-gfp-strong} by integrating with respect to $t$ over the interval $[0,1]$. 

Conversely, suppose formula \eqref{eq:ibp-gfp} holds for all $f\in\mathcal{S}$. Let $u \in (0,1]$ be chosen arbitrary and take $f \in \cS$ such that $f(X)$ depends on $X$ only through its path up to time $u$. Then, by definition of pathwise Malliavin derivative, we infer that $\D_t f(X) = 0$ for all $t>u$. Consequently, using the Fubini theorem again, \eqref{eq:ibp-gfp} becomes
	\[
	\int_0^u \E\left[X_t \D_t f(X)\right]\, \d t
	=
	\int_0^u \E\left[\int_0^1 K(t,s)\K^*\left[\D^2_{t,\,\cdot\,}f(X)\right](s)\, \d s\right]\, \d t.
	\]
	Since the latter identity holds for arbitrary $u \in (0,1]$, a direct application of fundamental Theorem of calculus ensures that the formula \eqref{eq:ibp-gfp-strong} takes place for Lebesgue almost every $t$, and for every $f\in \cS$. Finally, we note that under assumption \eqref{eq:strong_assumption} the  functions $t \in [0,1] \mapsto \E\left[X_t \D_t f(X)\right]$ and $$t \in [0,1] \mapsto \E\left[\int_0^1 K(t,s)\K^*\left[\D^2_{t,\,\cdot\,}f(X)\right](s)\, \d s\right]$$ belong to $\L^1(\d t)$. Now the rest of the proof follows similar lines as the proof of Theorem \ref{thm:ibp-gfp}. 
\end{proof}



\begin{cor}[Brownian Motion]\label{pro:ibp-bm}
The co-ordinate process $W$ satisfying assumption \eqref{eq:strong_assumption} is the Brownian motion if and only if
\begin{equation}\label{eq:ibp-bm}
\E\left[\int_0^1 W_t \,\D_t f(W)\, \d t\right]
=
\E\left[\int_0^1\!\!\!\int_0^t \D^2_{t,s} f(W)\, \d s \d t\right]
\end{equation} 
for all $f\in\cS$.
\end{cor}

\begin{proof}
The Brownian motion is a Gaussian Fredholm process with kernel $I(t,s) = \1_t(s)$.  The claim follows from this by noticing that $I(\d u,s) = \delta_s(\d u),$ (see Example \ref{exa:indicator_adjoint}), where $\delta_s$ is the unit mass at $s$.
\end{proof}


\begin{cor}[Gaussian Martingales]\label{cor:ibp-gm}
The co-ordinate process $M$ satisfying  assumption \eqref{eq:strong_assumption} is a Gaussian martingale with bracket $\la M \ra$ if and only if 
$$\label{eq:ibp-gm}
\E\left[\int_0^1 M_t \,\D_t f(M)\, \d t\right]
=
\E\left[\int_0^1\!\!\!\int_0^{t} \D^2_{t,s} f(M)\, \d\la M \ra_s\, \d t\right] 
$$ 
for all $f\in\cS$.
\end{cor}

\begin{proof}
By using a time-change, we observe that Gaussian martingales are Gaussian Fredholm processes with kernel $K(t,s) = I(\la M\ra_t, s)$. Consequently, $K(\d u, s) = \delta_{{\la M \ra}^{-1}_s}(\d u)$.  Therefore,
\begin{eqnarray*}
\int_0^1 K(t,s)\K^*\left[\D^2_{t,\,\cdot\,} f(M)\right](s) \,\d s 
&=&
\int_0^{\la M \ra_t}\!\!\!\int_0^1 \D^2_{t,u} f(M)\, \delta_{{\la M \ra}^{-1}_s}(\d u)\d s\\
&=&
\int_0^{\la M \ra_t} \D^2_{t,{\la M \ra}^{-1}_s}f(M)\, \d s,
\end{eqnarray*}
from which the claim follows by making a change-of-variables.
\end{proof}

\begin{cor}[Brownian Bridge]\label{cor:ibp-bb}
The co-ordinate process $B$ satisfying assumption \eqref{eq:strong_assumption} is the Brownian bridge if and only if
\begin{eqnarray}
\E\left[\int_0^1 B_t \D_t f(B)\, \d t\right]=
\E\left[\int_0^1\!\!\!\int_0^1\!\!\!\int_0^1 \left[\1_t(s) - t\right]\D^2_{t,u} f(B)\, \left[\delta_s(u) - \d u\right]\d s \d t\right] 
\label{eq:ibp-bb}
\end{eqnarray}
for all $f\in\cS$.  
\end{cor}

\begin{proof}
The integration-by-parts formula \eqref{eq:ibp-bb} follows from the orthogonal representation
$
B_t = W_t - t W_1
$
of the Brownian bridge. Indeed, we have
\begin{eqnarray*}
K(t,s) &=& \1_t(s) -t, \\
K(\d u, s) &=& \delta_s(u) - \d u.
\end{eqnarray*}
\end{proof}
\begin{rem}{\rm
The Brownian bridge also admits the so-called canonical representation
\[
B_t = \int_0^t \frac{1-t}{1-s} \, \d W_t,
\]
Consequently, we have
\begin{eqnarray*}
K(t,s) &=& \frac{1-t}{1-s}\1_t(s), \\
K(\d u, s) &=&
\frac{1}{1-s}\Big[(1-u)\delta_s(\d u) + \1_{u}(s)\d u\Big].
\end{eqnarray*}
It follows that an equivalent formulation for the integration-by-parts formula \eqref{eq:ibp-bb} is
\begin{eqnarray*}\nonumber
\lefteqn{\E\left[\int_0^1 B_t \D_t f(B)\, \d t\right] }\\
&=&
\E\left[\int_0^1\!\!\!\int_0^1\!\!\!\int_0^1 \frac{1-t}{1-s}\1_t(s)\D^2_{t,u} f(B)\, \left[\frac{1-u}{1-s}\delta_s(\d u) + \frac{\1_{u}(s)}{1-s}\d u\right]\d s \d t\right] \nonumber \\ 
&=&
\E\left[\int_0^1\!\!\!\int_0^t\frac{1-t}{(1-s)^2} \int_0^1\D^2_{t,u} f(B)\, \left[(1-u)\delta_s(\d u) + \1_{u}(s)\d u\right]\d s \d t\right] \nonumber \\
&=&
\E\left[\int_0^1\!\!\!\int_0^t\frac{1-t}{(1-s)^2} \left[(1-s)\D^2_{t,s} f(B) + \int_0^s \D^2_{t,u} f(B)\,\d u\right]\d s \d t\right] \nonumber. 
\label{eq:ibp-bb2}
\end{eqnarray*}
}
\end{rem}
\begin{rem}{ \rm
Corollary \ref{cor:ibp-bb} can be further extended to generalized bridges with respect to a general class of Gaussian processes by using the representation results of \cite{Sottinen-Yazigi-2014}.
}
\end{rem}


\section{Connection with the Abstract Wiener Space Approach of \cite{Shih-2011}}\label{sect:discussion}

In this part, we will discuss the link between our results and the integration-by-parts formula of Shih \cite{Shih-2011}.
For simplicity, we consider only the case of Brownian motion.

Let us denote by $\B$ the space $\mathcal{C}_{0}([0,1])$ of continuous functions on $[0,1]$, vanishing at zero. Let $X=W=(W_{t})_{t\in [0, 1]}$ be the standard Brownian motion. Then $\K = \I$ is just the integral operator and $\I^*$ is the identity operator (see Example \ref{exa:indicator_adjoint}).  The integrand space is $\mathcal{I}=\mathcal{L}^2$ and the Cameron--Martin space is $\mathcal{R}=\I\mathcal{L}^2$.  
It is well-known that $\I\colon\mathcal{I}\to\mathcal{R}\subset\mathcal{B}$ embeds $\mathcal{I}$ densely into $\mathcal{B}$.  Consequently, $(\I, \mathcal{R},\mathcal{B})$ is an abstract Wiener space in the sense of Gross \cite{Gross-1967a}.

Now, the pathwise Malliavin derivative introduced in Section \ref{sec:differentiation} (which coincides with the standard Malliavin derivative, see Lemma \ref{lmm:D=D}) satisfies
$$
\label{31d-1}
\langle \D f(x), y\rangle _{\L ^ {2}} = \nabla _{\I y} f(x)  
$$
for every $f\in \mathcal{S}, x, y \in \L ^ {2}$.  It can be shown that $\nabla _{\I  y} f(x)$ coincides with the Gross $\mathcal{R}$-derivative of $f(x)$ at $\I y$, see \cite{Shih-2011}, page 1241 or \cite{Gross-1967a}.

In \cite{Shih-2011}, the following characterization of Gaussian measures on $\B$ was obtained: if  $X$ is  a $\B$-valued random variable, then $\P$ is a  Gaussian measure  if and only if
\begin{equation}\label{31d-3}
\E\left[{\la X, \D f(X)\ra}_{\B,\B^*}\right] = \E\left[{\Tr}_\CM \D^2 f(X)\right]
\end{equation} 
for all $f:\B\to\mathbb{R}$ such that $\D^2f(X)$ is trace-class on $\CM$. Here the notation $\langle \cdot, \cdot, \rangle _{ \B, \B ^ {\ast} } $   means the usual dual pairing and ${\Tr}_\CM \D^2 f(X)$ is the trace of the Malliavin derivative $\D ^ {2}$ (also called the Gross Laplacian). 

Let us discuss the connection between our result in Corollary \ref{pro:ibp-bm} and the above formula (\ref{31d-3}). We will formally compute the left-hand side of (\ref{31d-3}). Let $\dot{W} $ be the so-called white noise, which is formally defined as a Gaussian process with covariance $\E[\dot{W}_{t} \dot{W}_{s}] =\delta (t-s)$. Recall that for every $g \in \L ^ {2}$, integrals of the form $\int_{0} ^ {1} g(s) \dot{W}_{s}\d s $ are well-defined centered Gaussian random variables. Also recall the formula that links the dual  pairing $\B -\B ^ {\ast}$  (recall that $\B ^ {\ast}$ is the space of signed measures) to the scalar product in  $\L^ {2} $ (see e.g. \cite{Gross-1967a}, page 1241):
\begin{equation}\label{31d-4}
{\langle \I x, h \rangle}_{\B, \B ^ {\ast}} = {\langle x, \I ^ {\prime}h \rangle}_{\L ^ {2}}, 
\end{equation}
for any $x\in \L ^ {2}$ and $h \in \B ^{\ast} $, where $\I ^ {\prime} $ is the injection from $\B ^ {\ast} $ into $( \L ^ {2} ) ^ {\ast}\simeq \L ^ {2}$ given by (see e.g. \cite{Ustu}, Chapter 1)
$$
\I ^ {\prime} h (t)= \int_{t} ^ {1} h(\d s).
$$
Using (\ref{31d-4}), the left-hand side of (\ref{31d-3}) can be expressed as follows: by setting $x=\dot W$ and $h(\d u)= \d\D_{u} f(W)$, we obtain
\begin{eqnarray*}
\E\left[\langle W, \D f(W) \rangle_{\B, \B ^ {\ast}}\right]
&=&\E\left[\int_{0} ^ {1} \dot{W}_{t} \left(\int_{t} ^ {1} h(\d u) \right) \d t\right] \\
&=& \E\left[\int_{0}^ {1}  W_{u} \, h(\d u)\right] \\
&=& \E\left[\int_{0}^ {1}  W_{u} \, \d\D_u f(W)\right], 
\end{eqnarray*}
which does not coincide with the left-hand side of (\ref{eq:ibp-bm}). Therefore, our formula in Corollary \ref{pro:ibp-bm} is different from the Shih's formula (\ref{31d-3}).

\bibliographystyle{siam}
\bibliography{pipliateekki}
\end{document}